\documentclass[12pt]{amsart}
\usepackage{xspace,amssymb,amsfonts,euscript,eufrak,mathrsfs}
\usepackage[all]{xy}
\usepackage[margin=2.8cm,footskip=25pt,headheight=20pt]{geometry}
\usepackage{fancyhdr}
\IfFileExists{srcltx.sty}{\usepackage{srcltx}}

\usepackage{amsthm,amsmath}
\usepackage{nicefrac}
\usepackage{graphicx,tikz}
 \usepackage{epstopdf,ifpdf}
\ifpdf
  \usepackage{pdfsync}
\fi
\usepackage{pxfonts,bbm}

\input xy
\xyoption {all}

  \newcommand{\nc}{\newcommand}
  \newcommand{\renc}{\renewcommand} 

\usepackage[latin1]{inputenc}

\def\wt{\widetilde}
\def\ov{\overline}

\def\to{\rightarrow}

\nc{\Br}{\mathcal{B}}
\nc{\id}{id}
\renc{\P}{\mathbb{P}}
\renc{\O}{\mathcal{O}}
\nc{\N}{\mathbb{N}}
\renc{\H}{\mathcal{H}}
\nc{\CC}{\mathcal{C}}

\nc{\F}{\mathcal{F}}
\nc{\G}{\mathcal{G}}
\nc{\Fq}{\mathbb{F}_q}
\nc{\Fqn}{\mathbb{F}_{q^n}}
\nc{\Q}{\mathbb{Q}}
\nc{\Ql}{\mathbb{Q}_\ell}
\nc{\Qlb}{\mathbbm{k}} 
\nc{\tG}{\wt G}

\nc{\CS}{\mathscr{C\!S}_{\!\!\mathbbm{1}}}
\nc{\IndGT}{\iIC_{1}}
\nc{\Ind}{\operatorname{Ind}}
\nc{\Sym}{\operatorname{Sym}}
\nc{\Ext}{\operatorname{Ext}}
\nc{\uk}{\underline{k}}
\nc{\triright}{\stackrel{[1]}{\to}}

\nc{\Ga}{\mathbb{G}_a} 
\nc{\Gm}{\mathbb{G}_m} 

\nc{\Loc}{\mathcal{L}}
\nc{\gG}{\Gamma}

\nc{\kos}[2]{\EuScript{K}_{#1,#2}}
\nc{\bet}{b}
\nc{\phc}{\Phi}
\nc{\vp}{\varphi}
\nc{\betT}{b_T}
\nc{\de}{\delta}
\nc{\QT}{Q}
\nc{\HT}{S}
\nc{\KM}{\EuScript{A}}
\nc{\ep}{\epsilon}
\nc{\Bi}{\mathbf{i}}
\nc{\BB}{B\times B}
\nc{\TT}{T\times T}

\nc{\BD}{B_\Delta}
\nc{\GD}{G_\Delta}
\nc{\C}{\mathbb{C}}
\nc{\R}{\mathbb{R}}
\nc{\Cs}[1]{\underline{\Qlb}_{#1}}
\nc{\IH}{I\!H}
\nc{\IC}{\mathbf{IC}}
\nc{\iIC}{\mathbf{K}}
\nc{\Gw}{G_w}
\nc{\up}{\vp^G_H}
\nc{\Kw}{K_w}
\nc{\B}{\mathcal{B}}
\nc{\SU}[1]{\mathrm{SU}(#1)}
\nc{\SL}[1]{{SL}_{#1}}
\nc{\GL}[1]{\mathrm{GL}(#1)}
\nc{\HU}[1]{\mathbb{H}\mathrm{U}(#1)}
\nc{\rank}{\mathrm{rk}}
\nc{\ft}{\mathfrak t}
\nc{\td}{\t^*}
\nc{\pur}[1]{\cF_\be^{#1}}
\nc{\Z}{\mathbb{Z}}
\nc{\fg}{\mathfrak g}
\nc{\tg}{\tilde \fg}
\nc{\HG}{H_G}
\renc{\k}{\mathbf{k}}
\nc{\Si}{\S i}
\nc{\hc}{\mathbb{H}^*}
\nc{\mc}{\mathcal}
\nc{\Hom}{\mathrm{Hom}}
\nc{\ti}{\tilde}
\renc{\O}{\mathcal {O}}
\nc{\hcBB}{\hc_{B\times B}}
\nc{\si}{\sigma}
\nc{\al}{\alpha}
\nc{\HH}{H\!H}
\nc{\Tor}{\mathrm{Tor}}
\nc{\KR}{\mc{KR}}
\nc{\Supp}{\mathrm{Supp}}
\nc{\ASu}{\mathrm{Supp}'}
\nc{\tri}{\tau}
\nc{\ext}{\mathrm{ext}}
\nc{\baet}[1]{\bar{e}(#1,t)}
\nc{\bht}[1]{\bar{h}(#1,t)}
\nc{\A}{\mathcal{A}}
\nc{\AH}{\A_H}
\nc{\AG}{\A_G}
\nc{\Lotimes}{\stackrel{L}{\otimes}}
\nc{\be}{\beta}
\nc{\Stosic}{Sto\v{s}i\'c\xspace}
\nc{\ga}{\gamma}
\nc{\BS}[1]{G_{#1}}
\nc{\BSi}{\BS\Bi}
\nc{\Bf}{\mathbf{f}}
\nc{\excise}[1]{}
\nc{\Bn}{\mathbf{n}}
\nc{\Ba}{\mathbf{a}}
\nc{\ubl}[1]{{P}_{#1}}
\nc{\ubr}[1]{{_{#1}}P}
\newcommand{\becircled}{\mathaccent "7017}
\nc{\oX}{\becircled X}
\nc{\cF}{\mathcal{F}}
\nc{\cG}{\mathcal{G}}
\nc{\om}{\omega}
\nc{\ublr}{\ub{\Bn}\times\ub{\Bn}}
\nc{\Sl}{S_\Bn}

\nc{\weight}{\mathrm{wt}\,}

\nc{\PP}{\mathbf{P}}

\nc{\D}{\mathbb{D}}

\nc{\End}{\mathrm{End}}
\nc{\codim}{\mathrm{codim}}
\nc{\Tr}{\mathrm{Tr}}
\nc{\dgmod}{\mathrm{dgMod}}
\nc{\res}{\mathrm{res}}
\nc{\ind}{\mathrm{ind}}
\nc{\ub}[1]{P(#1)}
\nc{\bsi}{\boldsymbol{\be}}
\nc{\Dbm}{D^b_{mix}}
\nc{\cE}{\mathcal{E}}
\nc{\cV}{\mathcal{V}}
\nc{\cW}{\mathcal{W}}
\nc{\FF}{\mathbb{F}}
\nc{\Hec}{\mathbf{H}}
\nc{\bFb}{\mathbf{F}^\bullet}
\nc{\bGb}{\mathbf{G}^\bullet}
\nc{\bM}{\mathbf{M}}
\nc{\Fr}{\mathrm{Fr}}
\renc{\AA}{\mathbb{A}}
\nc{\pH}{{}^p\mathcal{H}}

\nc{\Eul}{\EuScript{E}}
\nc{\K}{\EuScript{K}}
\nc{\modu}{\mathsf{mod}}
\DeclareMathOperator{\Spec}{Spec}
\DeclareMathOperator{\Res}{Res}

\nc{\pt}{\mathrm{pt}}
\nc{\gio}{\mathfrak{I}}
\nc{\gtw}{\mathfrak{S}}
\nc{\avg}{\operatorname{avg}}
\nc{\nGB}{R_W(q)}
\nc{\nB}{S_G(q)}

 \makeatletter 
\def\re
vddots{\mathinner{\mkern1mu\raise\p@ 
\vbox{\kern7\p@\hbox{.}}\mkern2mu 
\raise4\p@\hbox{.}\mkern2mu\raise7\p@\hbox{.}\mkern1mu}} 
\makeatother 
  \newtheorem{thm}{Theorem}
  \newtheorem{defi}[thm]{Definition}

  \newtheorem{prop}[thm]{Proposition}
  \newtheorem{cor}[thm]{Corollary}
  \newtheorem*{theorem*}{Theorem}

  \theoremstyle{remark}
  \newtheorem{remark}{Remark}
  \newtheorem{question}{Question}
  \nc{\helv}{\fontfamily{phv}\selectfont}
 
\begin{document}
\begin{center}
{\LARGE\bf  The geometry of Markov traces}
\vspace{10mm}

  \begin{tabular}{c@{\hspace{20mm}}c}
    {\sc\large Ben Webster}& {\sc\large Geordie Williamson}\\
   \it Department of Mathematics,&\it Mathematical Institute,\\ 
    \it University of Oregon &\it University of Oxford
 \end{tabular}
\vspace{1mm}

Email: {\helv bwebster@math.mit.edu}\\
{\helv geordie.williamson@maths.ox.ac.uk}
\vspace{1mm}

\end{center}
\medskip

{\small
\begin{quote}
  {\it Abstract.}
  We give a geometric interpretation of the Jones-Ocneanu trace on the
  Hecke algebra, using the equivariant cohomology of sheaves on $\SL
  n$.
  This construction makes sense for all simple algebraic groups, so we obtain
  a generalization of the Jones-Ocneanu trace to Hecke algebras of
  other types. We give a geometric expansion of this trace in terms of the irreducible characters of the Hecke algebra, and conclude that it agrees with a trace defined independently by Gomi.
  
  Based on our proof, we also prove that certain simple perverse
  sheaves on a reductive algebraic group $G$ are equivariantly formal for the conjugation action
  of a Borel $B$, or equivalently, that the Hochschild homology of any Soergel
  bimodule is free, as the authors had previously conjectured.

  This construction is closely tied to knot homology.  This
  interpretation of the Jones-Ocneanu trace is a more elementary manifestation
  of the geometric construction of HOMFLYPT homology given by the
  authors in
  \cite{WWcol}.
\end{quote}}

\renc{\thesubsection}{\arabic{subsection}}

\subsection{Introduction}
\label{sec:introduction}

The mid 1980's saw a remarkable burst of knot invariants appearing from surprising directions; one of the foremost was the HOMPLYPT polynomial, independently discovered by several groups of researchers (see \cite{HOMFLY,PT}). In \cite{Jon87} this invariant was then interpreted by Jones in terms of a trace, now known as the {\bf Jones-Ocneanu trace}, on the Hecke algebra $\Hec_n$. In this original paper, this trace is constructed using an inductive procedure, but Jones asks whether it has some direct interpretation in terms of the Kazhdan-Lusztig basis of \cite{KL79}.

In this paper, we describe a relationship between this trace on a Hecke
algebra and the geometry of the special linear group $G$, which passes through the ``Hecke category'' of certain equivariant sheaves on $G$. Since the Kazhdan-Lusztig basis can also be defined in terms of the Hecke category, this provides an answer to Jones' question:

\begin{theorem*}
  The Jones-Ocneanu trace applied to the Kazhdan-Lusztig basis element
  $C^{\prime}_w$ of a permutation $w$ is the mixed Poincar\'e polynomial of the $\BD$-equivariant   intersection cohomology of the Bruhat variety $\ov {BwB}$.  Equivalently, it is the
  bigraded dimension of the Hochschild homology of an indecomposible
  Soergel bimodule $S_w$.
\end{theorem*}
Here, $B$ denotes the subgroup of upper triangular matrices of $G$ and $\BD$ denotes that subgroup acting by conjugation on $G$.


This geometric construction works for reductive
algebraic groups of all types, providing a natural two-variable trace on the Hecke
algebra associated to any Dynkin diagram. Let $\Gamma$ be a Dynkin diagram, $G$ the corresponding split adjoint group over $\mathbb{F}_q$, $B$ a Borel subgroup, $W_{\Gamma}$ its Weyl group and $\Hec_{\Gamma}$ the Hecke algebra of $W$.

\begin{defi}
  For any Dynkin diagram $\Gamma$, we define the trace
  $\Tr_\Gamma:\Hec_\Gamma\to \Z(q^{\nicefrac 12})[t]$ to be the unique
  linear map sending a Kazhdan-Lusztig basis vector for $w\in
  W_\Gamma$ to the mixed Poincar\'e polynomial of the $\BD$-equivariant intersection cohomology of the Bruhat variety $\ov{BwB}$.
\end{defi}
Generalizing the
Jones-Ocneanu trace to other types has been an active avenue of
research for some years.  Traces satisfying a Markov-type condition on
Hecke algebras of classical type were classified by Geck and
Geck-Lambropoulou \cite{Geck,GL}, but such traces form an infinite
dimensional vector space.  Gomi \cite{Gomi} constructed a special
trace on the Hecke algebra of any finite type Coxeter group (even non-crystallographic). He constructs his trace by giving an explicit expression in terms of the irreducible characters of the Hecke algebra, and shows that it satisfies a strong Markov condition via case-by-case analysis.

Using the connection between the Hecke category on $G$ and the theory
of character sheaves, we prove our main result (for more precise statements
see Theorem \ref{main-theorem} and Corollary \ref{final}).  

\begin{theorem*}
  The geometrically defined trace $\Tr_\Gamma$ satisfies the following properties:
  \begin{itemize}
  \item It satisfies a Markov condition analogous to that of the Jones-Ocneanu
    trace.
  \item For Weyl groups, it coincides with the trace defined by Gomi \cite{Gomi}.
  \item Its expansion in terms of simple
    characters is given explicitly in terms of Molien series and
    Lusztig's Fourier transform.
  \end{itemize} 
\end{theorem*}

This gives a geometric interpretation of Gomi's
formula.
In fact, the decomposition of $\Tr_\Gamma$ on a Kazhdan-Lusztig basis element into ``almost characters'' of the Hecke algebra (the Fourier transform of the irreducible characters) is categorified by a decomposition of the $\BD$-equivariant intersection cohomology of $\ov{BwB}$ into summands whose mixed Poincar\'e series are given by Molien series. 
\medskip

In the course of our proof, we must show the following result, which is of some independent interest:
\begin{theorem*}
  Any simple $B\! \times \!B$-equivariant perverse sheaf on $G$ is
  $\BD$-equivariantly formal. Equivalently, the Hochschild homology of
  any Soergel bimodule is free.
\end{theorem*}
 This statement was previously proved in type $A$ by Rasmussen
\cite[Proposition 4.6]{Ras06} using an inductive calculation (similar to Jones's
construction of the trace), and conjectured for all types by the authors.
We give an independent, purely geometric proof using Lusztig's work on parabolic induction and restriction functors for character sheaves.\medskip

Just as this trace is intimately tied to the HOMFLYPT polynomial as
described in \cite{Jon87,HOMFLY}, its categorification is connected to
the triply graded link homology of Khovanov and Rozansky
\cite{Kho05,KR05,Ras06}. The main theorem of this paper in the case of
type $A$ is essentially equivalent to the fact that this knot homology
is a knot invariant categorifying the HOMFLYPT polynomial
(proven by Khovanov-Rozansky in \cite{KR05}).

However, we provide a description of
the categorified trace on the Hecke algebra from first principles,
separate from knot theory.  The authors relate this description to
knot theory in a separate paper \cite{WWcol}, which describes a
geometric construction of colored HOMFLYPT homology, a knot homology
theory whose construction had been proposed by Mackaay, \Stosic and Vaz
\cite{MSV}.

\subsection*{Acknowledgments}
\label{sec:acknowledgments}
We would like to thank Jean Michel for explaining Gomi's trace to us and asking if there is a relation with the Hochschild homology of Soergel bimodules. We would also like to thank Roman Bezrukavnikov, Victor Ginzburg, Hailong Dao and
Rapha\"el Rouquier for useful conversations and 
Jim Humphreys for comments on a previous version of this paper.
 BW was
supported by an NSF Postdoctoral Fellowship and NSA  Grant H98230-10-1-0199. GW was supported by an EPSRC Postdoctoral Fellowship.

\subsection{Hecke algebras}
\label{sec:hecke-algebras}

Let $\Gamma$ be a Dynkin diagram.  Attached to this Dynkin diagram, we have:
\begin{itemize}
\item An {\bf Artin braid group} $\Br_{\Gamma}$, which is generated by
  symbols $\si_i$ corresponding to vertices $i$ of $\Gamma$, with the relations
  \begin{equation*}
   \underbrace{\si_i\si_j\cdots}_{m_{ij} \text{ terms}}=\underbrace{\si_j\si_i\cdots}_{m_{ij} \text{ terms}}
  \end{equation*}
  where $m_{ij}$ is the Coxeter matrix:
  \begin{align*}
    m_{ij}=2&\quad \text{ if }\al_i^\vee(\al_j)\cdot\al_j^\vee(\al_i)=0, &
    m_{ij}=3&\quad \text{ if }\al_i^\vee(\al_j)\cdot\al_j^\vee(\al_i)=1,\\
    m_{ij}=4&\quad \text{ if }\al_i^\vee(\al_j)\cdot\al_j^\vee(\al_i)=2,&
    m_{ij}=6&\quad \text{ if }\al_i^\vee(\al_j)\cdot\al_j^\vee(\al_i)=3.
  \end{align*}
\item A {\bf Coxeter group} $W_{\Gamma}$ obtained as the quotient of
  $\Br_{\Gamma}$ by the relation $\si_i^2=1$, which is the Weyl group
  of the associated complex semi-simple Lie algebra.  We use $s_i$ to
  denote the image of $\si_i$ in $W_{\Gamma}$.
\item A {\bf Hecke algebra} $\Hec_{\Gamma}$, which is the quotient of the
  group algebra of the braid group $\Br_{\Gamma}$ over
  $\Z[q^{\nicefrac{1}{2}},q^{-\nicefrac{1}{2}}]$ by the relations
\begin{equation*}
  (\si_i-q^{\nicefrac{1}{2}})(\si_i+q^{- \nicefrac{1}{2}})=0.
\end{equation*}
This is a ``deformation'' of the relation $\si_i^2=1$, and the resulting
algebra is a flat deformation of the group algebra of $W_{\Gamma}$.
In fact, there is a
``standard'' basis $\si_w$ of $\Hec_{\Gamma}$ where the basis elements
are labeled by $w\in W_{\Gamma}$, which limits to the standard basis on
the group algebra.
\end{itemize}

Throughout we allow the empty Dynkin diagram $\emptyset$. By
convention $\Br_{\Gamma} = W_{\Gamma}$ is the trivial group and 
$\Hec(\emptyset)=\Z[q^{\nicefrac{1}{2}},q^{-\nicefrac{1}{2}}]$.

Given an inclusion of Dynkin diagrams $\Gamma'\to
\Gamma$, then there is an induced map for each of the algebraic
objects listed above.  In particular, for each infinite series of
classical groups $X_n$,
where $X=A,B,C,D$, we
have a tower of natural inclusions $\iota_n:\Hec(X_n)\to
\Hec(X_{n+1})$ for $n \ge 0$.  In all types other than $A$, this inclusion is
unambiguous.  In type $A_n$, we label all vertices with the integers
$1,\cdots, n$ from left to right, and take the inclusion
which preserves the labeling of generators. 

We first consider the case of the $A_n$ Dynkin diagrams.  We set
$\Hec_n = \Hec(A_n)$.
By work of Ocneanu and Jones \cite{Jon87}, it follows that
  the algebras $\Hec_n$ carry a unique system of traces\footnote{
Recall that if $B$ is an $A$-algebra (for a $A$ a commutative ring), and $M$ an $A$-module, then a
{\bf trace} on $B$ with values in $M$ is a map $\phi:B\to M$ of $A$-modules
such that $\phi(b_1b_2)=\phi(b_2b_1)$ for all $b_1, b_2 \in B$.}
$\Tr_n$ with values in $\Z((q^{\nicefrac{1}{2}}))[t]$
  normalized by the conditions:
  \begin{align*}
    \Tr_0(1) & = 1& \Tr_n(\si_{n}\iota_{n-1}(a))&= -t \cdot \Tr_{n-1}(a)\\ 
    \Tr_{n}(\iota_{n-1}(a))&=\frac{1+q^{\nicefrac 12}t}{1-q}\cdot
    \Tr_{n-1}(a)&\Tr_n(\si_{n}^{-1}\iota_{n-1}(a))&= q^{\nicefrac
      12} \cdot \Tr_{n-1}(a).
  \end{align*}
It is this trace and its generalizations which interest us.\medskip

The generalization of this theorem is due to Geck and Lambropoulou
\cite{GL}, but one must give more information to obtain a uniqueness
statement.

We use the labelings of the $B_n$ and $D_n$ Dynkin diagrams shown in Figure \ref{BandD}.
\begin{figure}
\begin{tikzpicture}
\node at (-4,.2) {\begin{tikzpicture}
\node (a) [circle,draw, inner sep=3pt,label=above:{$1$}] at (0,0){};
\node (b) [circle,draw, inner sep=3pt,label=above:{$2$}] at (1,0){};
\node (c) [circle,draw, inner sep=3pt,label=above:{$3$}] at (3,0){};
\draw (a.20) -- (b.160);
\draw (a.-20) -- (b.-160);
\draw (b.0) -- +(.4,0);
\draw (c.180) -- +(-.4,0);
\node at (2,0) {$\cdots$};
\draw (.4,.2) -- (.6,0) -- (.4,-.2);
\end{tikzpicture}};
\node at (4,0) {\begin{tikzpicture}
\node (aa) [circle,draw, inner sep=3pt,label=left:{$1$}] at (0,.5){};
\node (ab)  [circle,draw, inner sep=3pt,label=left:{$2$}] at (0,-.5){};
\node (b) [circle,draw, inner sep=3pt,label=above:{$3$}] at (1,0){};
\node (c) [circle,draw, inner sep=3pt,label=above:{$n$}] at (3,0){};
\draw (aa) -- (b);  \draw (ab) -- (b);
\draw (b.0) -- +(.4,0);
\draw (c.180) -- +(-.4,0);
\node at (2,0) {$\cdots$};
\end{tikzpicture}};
\end{tikzpicture}
\caption{The type $B$ and $D$ Dynkin diagrams}
\label{BandD}
\end{figure}
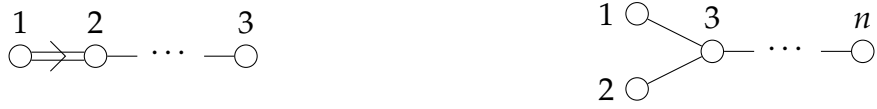
We let \begin{align*}
T_{n-1}&=\si_{n}\si_{n-1}\cdots \si_2 \si_1 \si_2^{-1}\cdots \si_{n}^{-1}\in \Hec(B_n)\\
U_{2n-1}&=\si_{2n}\cdots \si_3\si_1\si_2\si_3^{-1}\cdots \si_{2n}^{-1}\in \Hec(D_{2n}).
\end{align*}
\begin{defi} 
\label{markov-B}
  For each scalar $y \in \Z[q^{\pm \nicefrac 12} ,t]$, let $\Tr_n^y$ be the unique system of traces on the algebras $\Hec(X_n)$
  where $X_n=A_n,B_n,C_n$ or $D_n$ such that
  \begin{align*}
    \Tr^y_0(1) & = 1&
    \Tr_n^y(\si_{n}\iota_{n-1}(a))&=-t \cdot  \Tr_{n-1}^y(a)\\ 
    \Tr_{n}^y(\iota_{n-1}(a))&=\frac{1+q^{\nicefrac{1}{2}}t}{1-q}\cdot
    \Tr_{n-1}^y(a)&
    \Tr_n^y(\si_{n}^{-1}\iota_{n-1}(a))&=q^{\nicefrac 12} \cdot \Tr_{n-1}^y(a)\\
  \Tr_{n}^y(T_{n-1}\iota_{n-1}(a))& = y\cdot \Tr_{n-1}^y(a)&&\text{if } X=B,C\\
     \Tr_{2n}^y(U_{2n-1}\iota_{2n-1}(a))& = y\cdot \Tr_{2n-1}^y(a)&&\text{if } X=D
  \end{align*}
  (recall that by convention $\Hec_0=\Z[q^{\nicefrac{1}{2}},q^{-\nicefrac{1}{2}}]$.)
\end{defi}
The existence and uniqueness of such a trace was proved by Geck and Lambropoulou \cite{Geck,GL}.
\begin{remark}
  This is a stronger notion of a Markov trace than used by
  Geck-Lambropoulou, and stronger than necessary for topological
  purposes, but our trace will be of this form.
\end{remark}

\subsection{Hecke algebras and geometry}
\label{sec:hecke-algebr-geom}

Let $\FF_q$ denote a finite field with $q$ elements
and let $G$ denote a split semi-simple algebraic group over $\FF_q$
with fixed split maximal torus $T$ and Borel subgroup $B$. If $\Gamma$
denotes the Dynkin diagram of the root system of $G$, then we may
canonically identify $W_{\Gamma}$ with the Weyl group of $G$. Given any
$w \in W_{\Gamma}$ we denote by $\dot{w}$ a representative in $G$ for
$w \in W_{\Gamma}$. 
Consider the corresponding Bruhat decomposition
\[ G(\FF_q) = \bigsqcup_{w \in W_{\Gamma}} B(\FF_q) \cdot \dot{w} \cdot B(\FF_q). \] A classical
theorem of Iwahori identifies the convolution algebra of $B(\FF_q)
\times B(\FF_q)$-invariant $\C$-valued functions on $G(\FF_q)$ with $\Hec_{\Gamma}$, where
$q^{\nicefrac{1}{2}}$ is specialized at a fixed square root of $|\FF_q
| \in \C$. Under this identification the element $\si_i \in H_{\Gamma}$
corresponds to the characteristic function of the orbit $B(\FF_q)
\cdot \dot{s_i} \cdot B(\FF_q)$ multiplied by 
$q^{- \nicefrac{1}{2}}$.

The fact that the Hecke algebra arises as a convolution algebra
suggests a natural categorification using Grothendieck's
function-sheaf dictionary. Fix a prime $\ell$ different from the
characteristic of $\FF_q$ and let $\Qlb$ denote an algebraic
closure of the field of $\ell$-adic numbers.

We denote by $D^b_{B \times B}(G)$ the bounded $B \! \times \!B$-equivariant
derived category of mixed constructible $\Qlb$-sheaves
on $G$ (see, for example \cite{SGA4}, \cite{SGA4h}, \cite{BBD} and \cite{BL}), often referred to as the {\bf Hecke category}. There is a convolution product on $D^b_{B \times B}(G)$ which we denote by $\star$ (see \cite{SpIH}). We fix a square root
$q^{\nicefrac 12}$ of $q$ in $\Qlb$, which allows us to define a
square root $(\nicefrac 12)$ of the Tate twist $(1)$. 
We denote by $\langle m \rangle$ the shift-twist functor $[m](m/2)$ (note that $\langle m \rangle$ preserves weight).

\begin{defi} \label{def:ICnorm}
  Let $\IC_w$ denote the intersection cohomology complex corresponding
  to the $B \times B$ orbit $B\dot{w}B$, normalized so that the restriction
  of $\IC_w$ to $B\dot{w}B$ is $\underline{\Qlb}_{B\dot{w}B} \langle \ell(w)/2
  \rangle$.
\end{defi}

That is, we normalize intersection cohomology complexes so
  that the induced sheaf on $G/B$ is perverse and pure of weight 0.

Finally, let $D_{B \times B}(G)_0$ denote the subcategory of $D^b_{B
  \times B}(G)$ consisting of objects isomorphic to direct sums of
$\IC_w\langle m \rangle$ for $w \in W_{\Gamma}$ and $m \in \Z$ and
$\mathscr{K}_{B \times B}(G)_0$ its split Grothendieck group. We have:
\begin{itemize}
 \item The category $D_{B \times B}(G)_0$ is preserved under convolution (by the
   Decomposition theorem), and $\star$
  induces a ring structure on $\mathscr{K}_{B \times B}(G)_0$.
\item  The
  assignment $q^{1/2} \mapsto \langle {-1} \rangle $ makes
  $\mathscr{K}_{B \times B}(G)_0$ into an algebra over
  $\Z[q^{1/2},q^{-1/2}]$.
\end{itemize}
The following result is well-known:

\begin{thm}[See \cite{SpIH}]
\label{thm-heckeconv}
The Hecke algebra $\Hec_{\Gamma}$ can be identified with
$\mathscr{K}_{B \times B}(G)_0$. Under this identification the
Kazhdan-Lusztig basis element $C_w^{\prime}$ corresponds to the class of the
intersection cohomology complex $[ \IC_w ]$.
\end{thm}

For example, if $s$ is a simple reflection, then $C_s^{\prime} =
\sigma_s + q^{-\nicefrac 12}$ corresponds to
$\underline{\Qlb}_{P_s}\langle 1 \rangle$, the shift-twist of the
constant sheaf on the minimal parabolic subgroup $P_s = \overline{BsB}$.

The relationship between these sheaves and the Hecke algebra was an
important ingredient in the Kazhdan-Lusztig conjecture
\cite{KL79,KL80}.

Alternatively, one can consider the $B \! \times \! B$-equivariant cohomology of
$\IC_w$.  This is an indecomposable {\bf Soergel bimodule} 
$\hc_{\BB}(G;\IC_w)=S_w$ (for
more background on Soergel bimodules and their connections to
geometry, see the original paper
\cite{Soe92} of Soergel or authors' earlier paper \cite{WW}).

\subsection{A categorified trace}
\label{sec:categorified-trace}

Let $\BD\subset\BB$ be the diagonal subgroup.  Then we can associate
to $\IC_w$ and $S_w$ bigraded vector spaces $\hc_{\BD}\!(\IC_w)$ and
$\HH^*(S_w)$, the $\BD$-equivariant cohomology and Hochschild
homology respectively.

The bigrading on $\HH^*(S_a)$ is easy to describe: one grading, which
we call the {\bf $t$-grading}, is given by the homological grading on
$\HH^*$ and one is given by the
fact that $S_a$ is a graded bimodule over polynomials, with linear
polynomials in degree 2.  For reasons of knot theory,  we follow the convention of \cite{MSV} with respect to the grading on Hochschild homology: we normalize the gradings so that the differential on Hochschild homology has bidegree $(1,1)$.

The bigrading on $\hc_{\BD}\!(\IC_w)$ is of geometric origin, but less
obvious than the algebraic grading.  What we must use is the weight
grading on cohomology.  
By previous work of the authors \cite[Theorem
1.4]{WW}, there is an isomorphism $\hc_{\BD}\!(\IC_w)\cong \HH^*(S_w)$,
which sends the $q$-grading to the usual cohomological grading and the
$t$-grading to the weight plus cohomological grading, i.e.\ that given
by the norm of the eigenvalue of Frobenius.

Thus, we let
$\mathbb{H}_{\BD}^{i;j}(G;\IC_w)$ be subspace of
$\mathbb{H}_{\BD}^{i}(G;\IC_w)$ of weight $j$. In all situations
encountered below, Frobenius will always act semi-simply via powers of our fixed $q^{\nicefrac 12}$ and hence $\mathbb{H}_{\BD}^{i;j}(G; \IC_w)$ is equal to the $q^{\nicefrac{j}{2}}$-eigenspace of $\mathbb{H}_{\BD}^{i}(G; \IC_w)$ for the action of Frobenius.


\begin{defi} \label{defi-mPs}
Given a sheaf $\F \in D^b_{\BD}(G)$, its
  {\bf mixed Poincar\'e series} is
  \begin{equation*}
    P_{\F}(q,t)=\dim_{q,t}\hc_{\BD}\!(G;\F)=\sum_{i,j}\dim\mathbb{H}_{\BD}^{i;j}(G;\F)q^{\nicefrac{i}{2}}t^{j-i}.
  \end{equation*}
\end{defi}

We will always consider equivariant mixed Poincar\'e polynomials, and only specify the group if there is danger of confusion. Note that
\begin{equation} \label{shifttwist}
P_{\F\langle -i \rangle }(q,t) = q^{\nicefrac i2} P_{\F}(q,t).
\end{equation}

%

As before, $\Gamma$ is a Dynkin diagram, and we let
$\Gamma'=\Gamma-\{v\}$ for a single vertex $v$, and
$\iota:\Hec_{\Gamma'}\to \Hec_{\Gamma}$ is the induced inclusion.

\begin{thm}\label{main-theorem}
  The linear extension of the function
  $\Tr_\Gamma(C_w^{\prime})=\dim_{q,t}\hc_{\BD}\!(\IC_w) =
  P_{\IC_w}(q,t)$ only depends on the root system of $G$ and is a trace on
  $\Hec_{\Gamma}$. Furthermore, for each vertex $v$, we have
\begin{align*}
\label{norm1}  \Tr_{\emptyset}(1)&=1 &
\Tr_{\Gamma}(\si_{v}\iota(a))&= -t \cdot
\Tr_{\Gamma'}(a)\\ \Tr_{\Gamma}(\iota(a))&=\frac{1+ q^{\nicefrac 12}t}{1-q}\cdot
\Tr_{\Gamma'}(a) &
\Tr_{\Gamma}(\si_{v}^{-1}\iota(a))&=q^{- \nicefrac 12}\cdot \Tr_{\Gamma'}(a)
\end{align*}
In particular, when $\Gamma=A_n$, this trace coincides with that of
Jones-Ocneanu.  Furthermore, in the case where $\Gamma=B_n,C_n,D_n$,
it is the trace $\Tr_\Gamma^{-t}$ given in
Proposition \ref{markov-B}, where $y=-t$.
\end{thm}
This establishes that this trace pulled back to the braid group provides a knot invariant (up to normalization), called the HOMFLYPT
polynomial in the type $A$ case, and an invariant of knots in the
solid torus in type $B$, by the ``cylindrical Markov theorem'' (see \cite{GL}). 

\begin{proof}[Proof of Theorem \ref{main-theorem}]
We first show that our trace only depends on the root system of
the semi-simple group $G$, and not on its root datum. To this end let
\[
\phi : \widetilde{G} \to G
\]
denote a central isogeny and set $\widetilde{B} := \phi^{-1}(B)$. We
denote by $\widetilde{\IC}_w$ the intersection cohomology
complex on $\widetilde{G}$, normalized as in Definition
\ref{def:ICnorm}. We claim there is a canonical isomorphism:
\[
\hc_{{B}_{\Delta}}(G, \IC_w) \cong \hc_{\widetilde{B}_{\Delta}}(\widetilde{G}, \widetilde{\IC}_w).
\]

First note that $(\phi^*, \phi_*)$ give an equivalence
\[
D^b_{\widetilde{B} \times \widetilde{B}}(\widetilde{G})
\stackrel{\sim}{\to} D^b_{\widetilde{B} \times B}(G)
\]
and hence the
adjunction morphism $\F \to \phi_* \phi^* \F$ is an isomorphism for any $\F \in D^b_{\widetilde{B} \times B}(G)$. It
follows that
\[
\hc_{\widetilde{B}_{\Delta}}(G, \IC_w) \stackrel{\sim}{\to}
\hc_{\widetilde{B}_{\Delta}}(\widetilde{G}, \phi^*\IC_w)
= \hc_{\widetilde{B}_{\Delta}}(\widetilde{G}, \widetilde{\IC}_w)
\]
Lastly $\hc_{\widetilde{B}_{\Delta}}(G, \IC_w) = \hc_{{B}_{\Delta}}(G,
\IC_w)$ because $K = \ker \phi$ is a finite group and our coefficients are of
characteristic zero.

Our definition yields a trace since for sheaves $\F, \F^{\prime} \in D^b_{\BB}(G)$ we have
  \begin{equation*}
    \hc_{\BD}\!(G;\F\star\F')\cong \hc_{B\times B}(G\times G;\F\boxtimes\F')\cong \hc_{\BD}\!(G;\F'\star\F).
  \end{equation*}
where the $B \times B$-action on $G \times G$ in the equation above is given by
\[
(b_1, b_2) \cdot (g_1, g_2) = (b_1g_1b_2^{-1}, b_2g_2b_1^{-1}).
\]
It follows that $\Tr_{\emptyset}(1)=1$ since if $\Gamma = \emptyset$ then $G$ is the trivial group.

It remains to verify the behaviour of our trace under an inclusion
$\iota$. To do this we give a geometric interpretation of $\iota$. 
By the above considerations we may assume that $G$ is of
adjoint type. Now let $\Delta \subset R$
denote the based root system of $(G,B,T)$.  The inclusion $\Gamma
\subset \Gamma'$ determines a based root system $\Delta' \subset R'$ with an inclusion into  $\Delta \subset R$
such that $\Delta = \Delta' \cup \{ \alpha \}$ where $\alpha$ is the
simple root corresponding to the new vertex $v$ of $\Gamma$.

Now let $G'$ denote the closed subgroup of $G$ generated by root
subgroups and cocharacters corresponding to roots in $R'$, and let
$\widetilde{G'} = G'T$. Then $G'$ is a semi-simple algebraic group
with root system $R'$ and $\widetilde{G'} \cong G' \times \Gm$,
because we have assumed that $G$ is adjoint.
If we set \[
T' := T
\cap G', \qquad \widetilde{T'} := T, \qquad B' := B \cap G'\qquad  \text{and}\qquad \widetilde{B'}
:= B \cap \widetilde{G'}
\] then $T'$ and $\widetilde{T'}$ (respectively
$B'$ and $\widetilde{B'}$) are split maximal tori (respectively Borel
subgroups) of $G'$ and $\widetilde{G'}$.

  The projection and inclusion maps
\[
 G' \stackrel{p}{\leftarrow} \widetilde{G'}
 \stackrel{i}{\hookrightarrow} G
\]
induce functors $$p^* \colon D^b_{B' \times B'}(G) \to D^b_{\widetilde{B'} \times
  \widetilde{B'}}(G')\qquad\text{ and }\qquad i_* \colon D^b_{\widetilde{B'} \times
  \widetilde{B'}}(G') \to D^b_{\widetilde{B'} \times
  \widetilde{B'}}(G).$$ We define
\begin{eqnarray*}
 \gio_* : D^b_{B' \times B'} (G') \longrightarrow & D^b_{B \times B} (G) \\
\F \longmapsto & \ind^{B \times B}_{\widetilde{B'} \times
  \widetilde{B'}}i_* p^* \F.
\end{eqnarray*}
If $w' \in W_{\Gamma'}$ and $w$ denotes its image under the natural inclusion
$W_{\Gamma'} \hookrightarrow
W_{\Gamma}$ then one has $\gio_*( \IC_{w'}) \cong \IC_w$ because
$B/\widetilde{B'}$ is isomorphic to an affine space. 
Hence $\gio_*$ categorifies the inclusion $\iota : \Hec_{\Gamma'} \to
\Hec_{\Gamma'}$ under the identifications in Theorem
\ref{thm-heckeconv}.

Now, let $P\subset G$ be the minimal standard parabolic corresponding
to $\alpha \in \Delta$. Given $\F \in D^b_{B' \times B}(G')_0$  it remains to calculate the Poincar\'e series of the cohomology groups 
\[
\hc_{B_{\Delta}}(\Cs{B}\star \gio_* \F) \cong \hc_{B_{\Delta}}(\gio_* \F)\qquad 
\text{ and }\qquad 
\hc_{B_{\Delta}}(\Cs{P}\star \gio_* \F)
\]
in terms of $\hc_{B'_{\Delta}}( \F)$. It will be easier to work instead with
$T_{\Delta}$-equivariant cohomology, which is equivalent because
$B_{\Delta}/ T_{\Delta}$ is isomorphic to an affine space.

First note that, as $\widetilde{G'} \cong G' \times \Gm$ we have, by
the K\"unneth formula,
\begin{equation}\label{eq:Gm}
\hc_{T_{\Delta}}(\gio_* \F) \cong \hc_{T'_{\Delta}} (\F) \otimes \hc_{\Gm}(\Gm)
\end{equation}
(where $\Gm$ acts trivially on $\Gm$ in the second factor).

It remains to calculate $\hc_{B_{\Delta}}(\Cs{P}\star \gio_* \F)$.
Set $Z := (\ker \alpha)^0$ and let $L = C_G(Z)$ be its centralizer in
$G$. We can write $T \cong Z \times T_{v}$, with $\Gm \cong T_v$ via
the fundamental coweight corresponding to $\alpha$ (which belongs to
the cocharacter lattice because $G$ is of adjoint type). 

Now we have a Levi decomposition $P = LU$, where $U$ denotes the
unipotent radical of $P$. Hence we have a $T \times T$-equivariant
map $P \to L$ with fibers isomorphic to affine spaces. Moreover, $H :=
L/Z$ is a semi-simple group of rank 1. It follows that we have an acyclic $T
\times T$-equivariant map
\[
P \times_{T} \widetilde{G'} \to H \times G'.
\]
Moreover, the induced conjugation actions of $Z$ (resp. $T_v$) on $H$
(resp. $G'$) are trivial.

Hence
\begin{equation} \label{eq:SL2}
\hc_{T_{\Delta}}(G, \Cs{P}\star \gio_* \F)
\cong \hc_{(T_v)_\Delta}(H)  \otimes \hc_{Z_{\Delta}}(G', \F)
\end{equation}

Taking Poincar\'e series in the equations (\ref{eq:Gm},\ref{eq:SL2}), we get the relations
\begin{gather*}
    \Tr(\iota(a))=
\dim_{q,t} \hc_{B_{\Delta}}(\gio_* \F)
=\frac{1 + q^{\nicefrac 12}t}{1-q}\Tr(a)\\
\Tr(C^{\prime}_v\iota(a)) =
\dim_{q,t}\hc_{T_{\Delta}}(G, \Cs{P}\star \gio_* \F \langle 1 \rangle)
=\frac{q^{-\nicefrac{1}{2}} + qt}{1-q}\Tr(a)
\end{gather*}
This establishes one of our conditions.  For the others, we first
observe that $\si_v=C^{\prime}_v-q^{- \nicefrac{1}{2}}$ and $\si_v^{-1}=C^{\prime}_v-q^{\nicefrac{1}{2}}$.  This shows that
\begin{align*}
  \Tr(\si_v\iota(a))&=\frac{qt -t}{1-q}\cdot \Tr(a)&\Tr(\si_v^{-1}\iota(a))& = \frac{ q^{-\nicefrac 12} -q^{ \nicefrac 12 } }{1 - q}
  \cdot \Tr(a)\\
  &=-t\cdot \Tr(a)&&=q^{- \nicefrac 12}\cdot \Tr(a).
\end{align*}

Finally, we must establish our claim that
\begin{equation*}
  \Tr_{B_{n+1}}(T_{n}\iota(a)) = -t \cdot \Tr_{B_n}(a)
\end{equation*}
This follows from Corollary~\ref{final}, which shows that our trace
coincides with that considered by Gomi and from \cite[\S 4.4 \& 
4.5]{Gomi} where Gomi shows that his trace corresponds to that
constructed by Geck and Lambropoulou, which has this property by \cite[Proposition 4.5]{GL}.
\end{proof}

\begin{question} The above proof uses the Gomi's algebraic calculations to deduce that our trace satisfies all the conditions in Definition \ref{markov-B} for $y = -t$. We do not know whether the extra conditions in types $B$, $C$ and $D$ have a geometric explanation.
\end{question}

\subsection{Character sheaves and character expansions}

Any trace on a semi-simple algebra can be
expanded as a direct sum of the traces on irreducible representations,
i.e. as a sum of irreducible characters.  A formula for this was given
in type A by Ocneanu \cite{Wenzl}, in type B by Orellana
\cite{Orellana}, and in arbitrary type by Gomi \cite{Gomi}.  Using
our geometric perspective, we can give a much simpler proof than that
of \cite{Gomi}, which depended on case-by-case analysis.


Our proof is founded on the fact that the expansion of the Hecke
algebra as a sum of matrix algebras has a geometric manifestation:
given a simple perverse sheaf $\IC_w$, we can apply the ``averaging''
functor $\ind^{\GD}_{\BD}\colon D^b_{\BD}\!(G)\to D^b_{\GD}\!(G)$.
Since $G/B$ is projective, each summand
of $\iIC_w=\ind^{\GD}_{\BD}\IC_w$ is a shift of a simple perverse
sheaf by the Decomposition Theorem.
\begin{defi}
  A simple perverse sheaf in $D^b_{\GD}\!(G)$ is called a {\bf
    unipotent character sheaf} if it appears as a summand of $\iIC_w$
  for some $w$.  Following \cite[\S 5.2]{MSp}, we let $\widehat G(1)$
  denote the set of unipotent character sheaves.
\end{defi}

We choose our normalization so that each $\cV \in \widehat G(1)$
extends a local system of weight 0 in degree 0. Let $\widehat W$ denote the set of irreducible characters of $W$. For each
$\chi \in \widehat W$ there exists a unipotent character sheaf
$\cV_{\chi} \in G(1)$ whose restriction to the regular semi-simple
elements of $G$ is a local system with monodromy given by
$\chi$. Let us denote by $\widehat{ G}(1)_{ex}$ those
character sheaves in $\widehat G(1)$ which are not isomorphic to a character sheaf of the form
$\cV_{\chi}$ for $\chi \in \widehat W$.

Since induction is adjoint to restriction, we have a natural isomorphism
\begin{equation} \label{eq:ind-res}
  \hc_{\GD}\!(G;\iIC_w)\cong \hc_{\BD}\!(G;\IC_w).
\end{equation}
One might hope that the decomposition of $\iIC_w$ into simple character sheaves categorifies the formula expressing the trace of $C^{\prime}_w$ in terms of  the irreducible characters of the Hecke algebra. If this were the case then, by \eqref{eq:ind-res} we could write the weights of our trace in terms of the mixed Poincar\'e polynomials of simple character sheaves. We will see below that this is true after performing Lusztig's Fourier transform.

For $\cV\in \widehat G(1)$, we let $P_\cV(q,t)$ denote the
$G_\Delta$-equivariant mixed Poincar\'e polynomial of $\cV$.  Let
$\{\chi^{\prime} ,\chi\}$ be the matrix coefficients of Lusztig's Fourier transform.
\begin{prop}\label{weight}
 The weight $\varpi_\chi$ of the
  character $\chi$ in $\Tr$ is
  \begin{equation*}
    \varpi_\chi=\sum_{\chi^{\prime} \in \widehat W }
    \{\chi^{\prime},\chi\}\cdot P_{\cV_{\chi^{\prime}}}(q,t)
  \end{equation*}
\end{prop}
\begin{proof}
%

%
For the course of the proof consider $\mathscr S$ the split Grothendieck group of $D^b_{G_{\Delta}}(G)$. If $\F \in D^b_{G_{\Delta}}(G)$ we denote by $[\F]$ its class in $\mathscr S$ and, given a polynomial $a = a_i q^{\nicefrac{i}{2}} \in \mathbb{Z}[q^{\pm \nicefrac{1}{2}}]$ we set
\[
a \cdot [\F] := \sum_i a_i [\F \langle -i \rangle ] \in \mathscr S.
\]
In this notation, combining Lusztig's formula \cite[14.11]{LusCSIII} and \cite[23.1]{LusCSV} we obtain
\begin{equation} \label{lus-Kw}
[\iIC_w] = \sum_{\chi, \chi^{\prime} \in \widehat W} \chi_q (C_w^{\prime}) \{ \chi^{\prime}, \chi \} \cdot [\cV_{\chi'}]  +  \sum_{\cV \in \widehat{ G}(1)_{ex}} a_{\cV} \cdot  [\cV].
\end{equation}
Some remarks are in order about why our formula is equivalent to Lusztig's:
\begin{itemize}
  \item  Firstly, note that for Lusztig's objects $\bar K_{w}$ we have $\bar
  K_{w} \langle \ell(w)\rangle \cong \iIC_w$ and Lusztig normalizes
  character sheaves so as to be perverse.  
\item Secondly, note that
  \cite[14.11]{LusCSIII} is an equality after forgetting mixed
  structures, however this can be lifted to the mixed setting by using
  the fact that $\iIC_w$ is pure of weight zero. 
\item Thirdly, the sign
  $\epsilon_{\cV_{\chi}}$ (see \cite[13.10]{LusCSIII}) is equal to
  $(-1)^{\dim G}$ as $A_{\chi}$ occurs as a direct summand of the
  Springer sheaf $\iIC_{\id}$.
\end{itemize}

We show in Proposition \ref{formal} below that the $G_{\Delta}$-equivariant cohomology of any $\cV \in \widehat G(1)_{ex}$ is zero. Hence, applying hypercohomology to \eqref{lus-Kw} we obtain (note also \eqref{shifttwist})
\[
\Tr(C^{\prime}_w) =P_{\IC_w}(q,t)=P_{\iIC_w}(q,t)
=\sum_{\chi\in \widehat W} \chi_q(
C_w^{\prime} )\Big(\sum_{\chi^{\prime} \in \widehat
   W}\{\chi^{\prime},\chi\} \cdot P_{\cV_{\chi^{\prime}}}(q,t)\Big)
   \]
   and the proposition follows.
\end{proof}




Thus,
the primary point remaining to us is the computation of the mixed
Poincar\'e polynomial $P_\cV(q,t)$.  Our first step is to show that a
number of character sheaves have trivial cohomology.

If $L\subset G$ is the Levi of a parabolic, then there are adjoint
{\bf parabolic restriction} and {\bf parabolic induction} functors
\begin{equation*}
  \xy
(0,0)*{D^b_{L_\Delta}(L)}="A";
(40,0)*{D^b_{\GD}\!(G)}="B";
{\ar@/_12pt/_{\Ind^G_L} "A";"B"};
{\ar@/_12pt/_{\Res^G_L} "B";"A"};
\endxy
\end{equation*}
relating the conjugation-equivariant derived categories of these groups, defined by Lusztig \cite{LusCSI}.  These functors are different from the usual restriction and induction functors on equivariant derived categories, which we write uncapitalized.

For each character sheaf $\cV$, there is up to conjugacy a unique minimal Levi $L_\cV$ such that $\cV$ is a summand of $\Ind^G_L\cV'$ where $\cV'$ is a strongly cuspidal character sheaf on $L$.  This is also the minimal Levi for which $\Res^G_L\cV\neq 0$.

In the equivariant derived category $D_{\GD}\!(G)$, these Levi subgroups index a block decomposition of the category generated by character sheaves.   That is,
\begin{prop}\label{block}
  If $\cV,\cW$ are character sheaves and
  $\Ext_{\GD}^\bullet(\cV,\cW)\neq \{0\}$, then $L_\cV$ is conjugate to $L_\cW$.   

In particular, if $L_\cV\neq T$, then $\hc_{\GD}\!(G;\cV)=\{0\}$ and $P_\cV(q,t)=0$.
\end{prop}
\begin{proof}
  This was proved in the non-equivariant situation by Lusztig in
  \cite[7.2]{LusCSII}, and the equivariant result follows immediately from the existence of the spectral sequence
  \[
  E_2^{p,q} = \Ext^q(\cV,\cW) \otimes H_{\GD}^{p}(pt) \Rightarrow \Ext_{\GD}^{p+q}(\cV,\cW). \qedhere
  \] \end{proof}

One simple corollary of Proposition \ref{block} is that:

\begin{thm}\label{formal}
  Any unipotent character sheaf is $\GD$-equivariantly formal.
\end{thm}
\begin{proof}
If $\cV$ is induced from a cuspidal sheaf on a non-abelian Levi, then by Proposition \ref{block}, $\hc_{G_\Delta}(\cV)=0$, so $\cV$ is equivariantly formal.

On the other hand, every character sheaf not covered by the
previous assertion is a summand of $\IndGT$, which is equivariantly formal, since $\IC_{1}=\Qlb_B$ is.
\end{proof}

This allows us to answer definitively a question answered positively for type A in a previous paper of the authors \cite{WW}.

\begin{cor}
  For any $G$ and any $w\in W$, the sheaf $\IC_w$ is
  $\BD$-equivariantly formal.
\end{cor}
We thank Hailong Dao for help with the following proof. 
\begin{proof}
The sheaf $\iIC_w$ is equivariantly formal, since it is a sum of unipotent character sheaves.  
Thus the cohomology $H^*_{\BD}(\IC_w)=H^*_{\GD}(\iIC_w)$ is free as a module over $H^*(B\GD)$; we wish to conclude that it is free as a module over the larger ring $H^*(B\BD)$. 

This follows from standard arguments in commutative algebra: since $\Spec H^*(B\BD)$ is smooth, it suffices to show that the depth of the stalk at each point is the dimension of the variety.  On the other hand, this follows immediately from Proposition 1.2.16 in \cite{BH}, and the fact that both $H^*(B\BD)$ and $H^*(B\GD)$ are smooth and thus Cohen-Macaulay.
\end{proof}

\subsection{Character sheaves and Molien series}
\label{sec:char-sheav-moli}

Proposition \ref{block} shows that in equation (\ref{weight}), we need only sum over the sheaves
that appear as summands of $\IndGT$, since the trivial local system is
the only unipotent character sheaf on $T$.  Let us give a more careful description of these sheaves.

Consider the variety $\tG=G_\Delta\times_{B_\Delta}B$ and let
$c:G_\Delta\times_{B_\Delta}B\to G$ be the map induced by
$(g,t)\mapsto gtg^{-1}$.  A unipotent character sheaf is induced from $T$ if and only if it is a summand of $\IndGT=c_*\underline{\Qlb}_{\tG}$.

It is known that the map $\tG\to G$ is small, and so the sheaf $\IndGT$ is the intermediate extension of a local system on the regular semi-simple locus $G_{reg}$, since $\tG_{reg}\to G_{reg}$ is a Galois covering with Galois group $W$.  Because intermediate extension preserves endomorphism rings, 
\[
\End(\IndGT)=\End_{\Qlb W}(\Qlb W)=\Qlb W.
\] 
Thus, the isomorphism types of simple summands which appear are in canonical bijection with simple representations of $W$. Given an irreducible character $\nu \in \widehat W$ recall that we denote by $\cV_{\nu}$ the corresponding unipotent character sheaf.

Set $\ft^* = X(T) \otimes_{\mathbb{Z}} \Qlb$, where $X(T)$ denotes the lattice of characters of $T$. Note that $\ft*$ is a $W$-module in a natural way, and set $E = \Sym^\bullet(\ft^*)\otimes
\wedge^\bullet\ft^*$, which we may view as a $W$-module with the tensor product action. We have a canonical identification $E \cong H^*_{T_\Delta}(T)$.
The geometric realization of this algebra equips it with a bigrading by degree and weight as in Section \ref{sec:categorified-trace}. Given an irreducible character $\nu \in \widehat W$ the graded dimension
\[
M_\nu(q,t)=\sum_{i,j}q^{i/2}t^{j-i}\dim\Hom_W(V_\nu,E)^{i,j}
\]
is called the {\bf Molien series} of $\nu$.
\begin{prop}
\begin{math}
     \displaystyle H^\bullet_{G_\Delta}(\cV_\nu)= \Hom_W( V_{\nu}, E).
\end{math}
\end{prop}

\begin{proof}
Because $K_1 = c_* \Qlb_{\widetilde{G}}$ we have canonical isomorphisms
\begin{equation} \label{eq:E1}
H^\bullet_{\GD}(\IndGT)\cong H^\bullet_G(\widetilde{G}) \cong H^\bullet_{\BD}(B)\cong H^{\bullet}_T(pt) \otimes H^{\bullet}(T) \cong E.
\end{equation}
The $W$-action on $\IndGT$ induces a $W$-action on $E$ such that with this induced action 
\[
H^\bullet_{\GD}(\cV_\nu,E)= \Hom_W (V_\nu, E).
\]  
Thus, we need only check that the induced $W$-action on $E$ via the the isomorphisms \eqref{eq:E1} is the natural action.  Since $W$ acts by algebra homomorphisms, it suffices to check this on $H^1$ and on pure classes in $H^2$. To this end, note that the restriction map
\begin{equation} \label{eq:reg-restrict}
H^\bullet_G(\widetilde{G}) \to H^\bullet_G(\widetilde{G}_{reg})
\end{equation}
is $W$-equivariant, where the $W$-action on $H_G(\widetilde{G}_{reg})$ is induced by the deck transformations of the covering $\widetilde{G}_{reg} \to G_{reg}$. We claim that the restriction map \eqref{eq:reg-restrict} is an injection on $H^1$ and an isomorphism on the subspace of pure classes.

We may identify $\widetilde{G}_{reg} = G/T \times T_{reg}$. Under this identification the $W$-action by deck transformations is given by $w \cdot (gT, t) = (g\dot{w}^{-1}T, \dot{w}t\dot{w}^{-1})$. Furthermore the $G$-action on $T_{reg}$ is trivial and hence
\begin{equation} \label{eq:E2}
H^\bullet_G(\widetilde{G}_{reg}) = H^\bullet_G(G/T) \otimes H^{\bullet}(T_{reg}) \cong H^\bullet_T(pt) \otimes H^{\bullet}(T_{reg}).
\end{equation}
It is straightforward to check that under the isomorphisms \eqref{eq:E1} and \eqref{eq:E2} the restriction map \eqref{eq:reg-restrict} is the tensor product of the identity on $H_T^{\bullet}(pt)$ and the restriction map $H^\bullet(T) \to H^\bullet(T_{reg})$ on ordinary cohomology.

Now $T_{reg}$ is obtained from the smooth variety $T$ by removing a divisor, and so $H^{\bullet}(T) \to H^{\bullet}(T_{reg})$ is injective on $H^1$ and there are no pure classes in positive degree in either group. Hence the pure classes in both \eqref{eq:E1} and \eqref{eq:E2} consists of the image of $H^{\bullet}_T(pt)$. It follows that \eqref{eq:reg-restrict} is injective on $H^1$ and an isomorphism on pure classes as claimed.

Finally, given the above description of the action of $W$ on $G/T \times T_{reg}$ one may conclude that the induced action of $W$ on \eqref{eq:E2} is the tensor product of the natural action of $W$ on $H_T(\pt) = \Sym^\bullet(\ft^*)$ and action induced by conjugation on $H^\bullet(T_{reg})$. The proposition then follows.
\excise{
For the latter, this follows by \cite[Theorem 7.2.2]{CG97}; for the former, a theorem of Ginzburg  \cite[Theorem 8.1]{Ginzburg-character} shows that the category of perverse sheaves Serre generated by $\IndGT$ is equivalent to the representations of the smash product $\Qlb[W]\# \Sym^\bullet \ft$ sending $\IndGT$ to $\Qlb[W]$ with trivial $\Sym^\bullet \ft$-action.  Thus, we have that \[H^1(\IndGT)=\Ext^1_{\mathsf{Perv}}(\Qlb_{G},\IndGT)=\Ext^1_{\Qlb[W]\#\Sym^\bullet h}(\Qlb,\Qlb[W])=\Ext^1_{\Sym^\bullet \ft}(\Qlb,\Qlb)=\wedge^\bullet\ft^*\] with the natural $W$-action. 
}
\end{proof}

This geometric realization of the Molien series allows us to arrive at
the main theorem of \cite{Gomi} by a very different route.
\begin{cor}\label{final}
  For all characters $\nu\in \widehat W$, we have
  $M_\nu(q,t)=P_{\cV_\nu}(q,t)$. In particular, by Proposition
  \ref{weight}, it follows that
  \begin{equation*}
    \Tr(x)=\sum_{\chi\in\widehat W}\chi_q(x)\cdot\{\chi,\nu\}\cdot M_\nu(q,t).
  \end{equation*}
\end{cor}

\begin{remark}\label{smash}
  The above identification of the cohomology of a character sheaf with a multiplicity space reflects a deeper fact about the structure of $\GD$-equivariant sheaves on $G$: by an extension of the arguments above, one can construct a natural isomorphism $\Ext^\bullet_{G_{\Delta}}(\IndGT)\cong \Qlb[W]\#E$ as bigraded algebras. This provides an equivariant version of a theorem of Ginzburg \cite[Theorem 8.1]{Ginzburg-character}. \end{remark}

\bibliography{./gen}
\bibliographystyle{amsalpha}

\end{document}